\documentclass[11pt]{article}
\usepackage{pslatex,times}
\usepackage{amsmath,amssymb,amsthm,url,color}
\usepackage{multirow}
\usepackage[normalem]{ulem}
\usepackage{tikz}
\usepackage{tkz-graph}
\usetikzlibrary{shadows}
\usetikzlibrary{patterns}

\newtheorem{theorem}{Theorem}
\newtheorem{lemma}[theorem]{Lemma}
\newtheorem{prop}[theorem]{Proposition}

\theoremstyle{definition}
\newtheorem{defn}[theorem]{Definition}
\newtheorem{example}[theorem]{Example}

\newcommand{\uu}{\mathbf{u}}

\newcommand{\xx}{\mathbf{x}}
\newcommand{\yy}{\mathbf{y}}

\renewcommand{\d}{\mathrm{d}}

\newcommand{\boxprod}{\,\Box\,}

\DeclareMathOperator{\diam}{diam}

\DeclareSymbolFont{mysymbol}{OMS}{cmsy}{m}{n}
\DeclareMathSymbol{\infty}{\mathord}{mysymbol}{49}

\title{Error-correcting codes from $k$-resolving sets}

\author{
Robert F.\ Bailey\footnote{School of Science \& Environment (Mathematics), Grenfell Campus, Memorial University of Newfoundland, Corner Brook, NL  A2H~6P9, Canada.  Email: \texttt{rbailey@grenfell.mun.ca}}
\and 
Ismael G.\ Yero\footnote{Departamento de Matem\'aticas, Escuela Polit\'ecnica Superior de Algeciras, Universidad de C\'adiz, Av.\ Ram\'on Puyol s/n, 11202 Algeciras, Spain.  Email: \texttt{ismael.gonzalez@uca.es}}
}

\begin{document}

\maketitle

\begin{abstract}
We demonstrate a construction of error-correcting codes from graphs by means of $k$-resolving sets, and present a decoding algorithm which makes use of covering designs.  Along the way, we determine the $k$-metric dimension of grid graphs (i.e.\ Cartesian products of paths).\\[1ex]


\end{abstract}

\section{Introduction}

\subsection{Error-correcting codes}
Error-correcting codes are applied to the accurate transmission and storage of data.  When information is received by a target, or read from a storage medium, errors may be introduced---for example, due to signal noise, or the medium being damaged---so to alleviate this problem, redundancy is introduced in order that the intended message can still be understood.  For an introduction to coding theory, see~\cite{Pless}.

Formally, an {\em error-correcting code} (or simply a {\em code}) is a collection $\mathcal{C}$ of vectors, called {\em codewords}, of given length $\ell$ over a fixed alphabet.  The {\em Hamming distance} between two codewords $\xx=(x_1,\ldots,x_\ell)$, $\yy=(y_1,\ldots,y_\ell)$ is the number of positions where they differ, i.e.\ $|\{i\, : \, x_i\neq y_i \}|$.  The {\em minimum distance} of $\mathcal{C}$ is the least Hamming distance between any two distinct codewords; if the minimum distance is $D$, then the {\em correction capability} of $\mathcal{C}$ is $r=\left\lfloor (D-1)/2 \right\rfloor$.   Suppose that a codeword $\xx$ is transmitted via a noisy channel which causes errors to appear, i.e.\ some symbols are replaced with others.  If there are $r$ errors or fewer, the received word has a unique nearest neighbour in $\mathcal{C}$, which is necessarily the transmitted word $\xx$.  For this to be useful in practice, an efficient decoding algorithm is needed to determine the nearest neighbour.

Traditionally, the most familiar error-correcting codes are {\em linear codes} (i.e.\ subspaces of vector spaces over finite fields)~\cite{Pless}, where the alphabet size is small (such as binary codes, which have an alphabet of size~2).  Other classes of codes include {\em permutation codes}~\cite{Cameron2010}, where each codeword is a permutation of $n$ symbols, so the length and alphabet size are both equal to $n$; codes with larger alphabet sizes have been the subject of more recent attention, in part because of applications such as powerline communications~\cite{Chu} and flash memory devices~\cite{TamoSchwartz10}.

\subsection{$k$-resolving sets}
We consider finite, simple, connected, undirected graphs.  The {\em distance} between two vertices $u$ and $v$ of a graph $G$ is the length of a shortest path between $u$ and $v$, and we denote this by $\d_G(u,v)$.  In recent years, much attention has been paid to the {\em metric dimension} of graphs: this is the smallest size of a subset of vertices (called a {\em resolving set}) with the property that the list of distances from any vertex to those in the set uniquely identifies that vertex, and is denoted by $\dim(G)$.

These concepts were introduced to graph theory in the 1970s by Harary and Melter~\cite{Harary76} and, independently, Slater~\cite{Slater75}; however, in the context of arbitrary metric spaces, the concept dates back at least as far as the 1950s~\cite{Blumenthal53}.  Various applications have been suggested for resolving sets and metric dimension of graphs, including combinatorial optimization \cite{SeboTannier04}, pharmaceutical chemistry \cite{Chartrand00}, robot navigation \cite{Khuller96} and sonar \cite{Slater75}.  For more information, see~\cite{bsmd,Chartrand00}.

The following definition is a natural generalization of the notion of resolving sets.

\begin{defn}
Let $G=(V,E)$ be a graph.  An ordered set of vertices $(v_1,\ldots,v_{\ell})$ is a {\em $k$-resolving set} for $G$ if, for any distinct vertices $u,w\in V$, the lists of distances $(\d_G(u,v_1),\ldots,\d_G(u,v_{\ell}))$ and $(\d_G(w,v_1),\ldots,\d_G(w,v_{\ell}))$ differ in at least $k$ positions.
\end{defn}

In the case $k=1$, we have the usual notion of a resolving set for $G$.  For $k>2$ it is not necessarily the case that an arbitrary graph $G$ has a $k$-resolving set; for example, a complete graph $K_n$ with $n\geq 3$ has a $2$-resolving set but not a $3$-resolving set.  If $G$ has a $k$-resolving set, we denote the least size of a $k$-resolving set by $\dim_k(G)$, the {\em $k$-metric dimension} of $G$.  A $k$-resolving set of size $\dim_k(G)$ is called a {\em $k$-metric basis} for $G$.  If $k$ is the largest integer for which $G$ has a $k$-resolving set, then we say that $G$ is a {\em $k$-metric dimensional} graph.

The notion of $k$-resolving sets in graphs was introduced in~\cite{yero1} and further studied in~\cite{yero3,yero4,yero2}; it was then extended to more general metric spaces in~\cite{BeardonRV}.

\section{Codes from $k$-resolving sets}
The two themes of this paper are tied together by the following definition.

\begin{defn} \label{defn:kRS}
Let $G$ be a graph with $n$ vertices and diameter $d$, and let $S=\{v_1,v_2,\ldots,v_{\ell}\}$ be a $k$-resolving set for $G$ of size $\ell$.  Then the set
\[ \mathcal{C}(G,S) = \{ (\d_G(u,v_1),\d_G(u,v_2),\ldots,\d_G(u,v_{\ell})) \, : \, u\in V \} \]
is called a {\em $(G,k)$-code}.
\end{defn}
It follows from the definition that $\mathcal{C}(G,S)$ is an error-correcting code of length $\ell$, size $n$ and minimum Hamming distance at least $k$, over the alphabet $\{0,\ldots,d\}$, which can correct $r=\lfloor (k-1)/2 \rfloor$ errors.  In order for $r$ to be non-zero, we require that $k\geq 3$ (otherwise the code has no practical purpose).

For $\mathcal{C}(G,S)$ to be used for error correction, we need a decoding algorithm.  Let $G_j(u)$ denote the subset of vertices of $G$ at distance $j$ from $u$.
Now suppose that $u\in V$ and $\uu=(\d_G(u,v_1),\d_G(u,v_2),\ldots,\d_G(u,v_{\ell})) \in \mathcal{C}(G,S)$; suppose that we transmit $\uu$ and receive the word $\xx=(x_1,x_2,\ldots,x_{\ell})$, which is assumed to have at most $r$ errors.

\begin{lemma} \label{lemma:uniqueness}
Let $u$, $\uu$ and $\xx$ be as above, and suppose that $I$ is an $(\ell-r)$-subset of $\{1,\ldots,\ell\}$.
\begin{itemize}
\item[(i)] If the received word $\xx$ contains no errors in the positions indexed by $I$, then
\[ \bigcap_{i\in I} G_{x_i}(v_i) = \{u\}. \]
\item[(ii)] If the received word $\xx$ does contain an error in a position in $I$, then
\[ \bigcap_{i\in I} G_{x_i}(v_i) = \varnothing. \]
\end{itemize}
\end{lemma}

\begin{proof} If there are no errors in the positions indexed by $I$, then both $\xx$ and $\uu$ contain the same entries in those positions, and thus $u$ is the unique vertex at those distances from the corresponding entries in the $k$-resolving set.  If, however, there are errors in those positions, then no such vertex can exist, and thus the intersection is empty.
\end{proof}

The goal, therefore, when decoding a $(G,k)$-code is to find (as quickly as possible) an $(\ell-r)$-subset of positions for which the intersection $\displaystyle{\bigcap_{i\in I} G_{x_i}(v_i)}$ is non-empty.  Successively enumerating the $(\ell-r)$-subsets of $\{1,\ldots,\ell\}$ will achieve this, but will be slow in practice.  However, if we assume that there are at most $r'<r$ errors (say $r'=2$ or $r'=3$), we can make use of the following idea.

\begin{defn} \label{defn:uncovering}
Let $\nu$, $\kappa$, $\tau$ be integers such that $\nu\geq \kappa \geq \tau \geq 0$.  A {\em $(\nu,\nu-\kappa,\tau)$-uncovering} is a collection $\mathcal{U}$ of $(\nu-\kappa)$-subsets of $\{1,\ldots,\nu\}$ with the property that any $\tau$-subset of $\{1,\ldots,\nu\}$ is disjoint from at least one member of $\mathcal{U}$.
\end{defn}

If we take the complements of each $(\nu-\kappa)$-subset in $\mathcal{U}$, we obtain a {\em $(\nu,\kappa,\tau)$-covering design}, which are much more widespread in the literature: see the survey by Mills and Mullin~\cite{MillsMullin92} for details of these.  Uncoverings were introduced by the first author in~\cite{btubb,ecpg} where they were applied to decoding permutation codes; the same concept was also devised under the name {\em antiblocking system} by Kroll and Vincenti~\cite{Kroll08} for a decoding algorithm for linear codes.  A further application to network reliability was given in~\cite{ubbnet}.

The best known bound on the minimum size of coverings (and thus uncoverings also) is known as the {\em Sch\"onheim bound}, proved in~\cite{Schonheim}.  It states that for given $\nu$, $\kappa$ and $\tau$, the least size of a $(\nu,\kappa,\tau)$-covering design is
\[ L(\nu,\kappa,\tau) = \left\lceil \frac{\nu}{\kappa} \left\lceil \frac{\nu-1}{\kappa-1} \left\lceil \cdots \left\lceil \frac{\nu-\tau+1}{\kappa-\tau+1} \right\rceil \cdots \right\rceil \right\rceil \right\rceil. \]
Covering designs meeting this bound are known (or known asymptotically) in many cases: see~\cite{handbook,MillsMullin92} for tables of results.  The database of best-known covering designs~\cite{lajolla} is useful for finding uncoverings with small parameters.

\begin{example} \label{example:cov-unc}
The following is an $(8,5,2)$-covering design:
\[
\begin{array}{ccccc}
4 & 5 & 6 & 7 & 8 \\
1 & 2 & 3 & 7 & 8 \\
1 & 4 & 5 & 6 & 8 \\
2 & 3 & 4 & 5 & 6
\end{array}
\]
By taking the complements of each block, we obtain an $(8,3,2)$-uncovering:
\[
\begin{array}{ccc}
1 & 2 & 3 \\
4 & 5 & 6 \\
2 & 3 & 7 \\
1 & 7 & 8
\end{array}
\]
It can easily be seen that any pair chosen from $\{1,\ldots,8\}$ is disjoint from at least one row.
\end{example}

\subsection{A decoding algorithm}

Suppose we have a $(G,k)$-code of length $\ell$, and we wish to correct $r'$ errors.  If $\mathcal{U}$ is an $(\ell, \ell-r, r')$-uncovering, we proceed as follows: for a received word $\xx$, we consider each $I\in\mathcal{U}$ and obtain $\displaystyle{\bigcap_{i\in I} G_{x_i}(v_i)}$.  By Lemma~\ref{lemma:uniqueness}, this intersection will either be empty or contain the vertex $u$ corresponding to the transmitted word $\uu$.  If $\mathbf{u}$ contains at most $r'$ errors, by the definition of uncovering we know that there exists an $I\in\mathcal{U}$ disjoint from the error positions; consequently, we are guaranteed to be able to find the transmitted word.
To compute $\displaystyle{\bigcap_{i\in I} G_{x_i}(v_i)}$, consider the matrix $M$ whose rows are indexed by $V$ and whose columns are indexed by $S$, and where the entries are $M_{uv}=\d_G(u,v)$  (so the rows of $M$ are precisely the codewords).  For a given $I\subseteq S$, examine the rows of the submatrix to find a row which agrees with $\xx$ in those positions; if such a row exists, by Lemma~\ref{lemma:uniqueness} it must be unique and correspond to the vertex $u$.

\subsection{Complexity}
The matrix $M$ is a submatrix of the distance matrix of $G$, which can be computed in $O(|V|^3)$ time (for instance, by the Floyd--Warshall algorithm; see~\cite[{\S}25.2]{intro-alg}); however, this need only be done once, prior to the implementation of the code.  For a given instance of the decoding problem, where the input is a received word $\xx$ and the output the transmitted word $\uu$, the matrix $M$ must be examined at most $|\mathcal{U}|$ times, and at most $|I|\cdot|V|$ steps are required each time.  Thus the overall complexity of the decoding algorithm is $O(|\mathcal{U}|\cdot|I|\cdot|V|)$.

\section{Some covering designs}

Consider a $(\nu,\kappa,\tau)$-covering design where $\nu$ is given by a linear function in $\kappa$.  Then for fixed $\tau$ there exists a threshold value $\kappa_0$ such that, beyond this threshold, the Sch\"onheim bound $L(\nu,\kappa,\tau)$ remains constant.
For example, if $\nu=2\kappa+3$ we see that for all $\kappa\geq 9$, we have $L(2\kappa+3,\kappa,2)=7$, while for all $\kappa\geq 21$, we have $L(2\kappa+3,\kappa,3)=15$.
Unfortunately, examples of covering designs which actually achieve this lower bound are rare.  However, while it would be desirable to have optimal coverings, from the perspective of complexity a family of coverings of constant size (for a given value of $\tau$) is an acceptable solution.  So the following construction (suggested by F.~Petrov\footnote{Personal communication via \texttt{mathoverflow.net}, February 2016.}) is very useful.

\begin{prop} \label{prop:petrov}
Let $\nu=a\kappa+b$ (where $a\neq 0$ and $b$ are fixed constants), and let $\tau$ be a fixed constant.  Then, provided $\kappa$ is sufficiently large, there is a $(\nu,\kappa,\tau)$-covering design of size bounded by a constant dependent only on $a$ and $\tau$.
\end{prop}

\begin{proof} Let $m=\lfloor \kappa/\tau \rfloor$ and $s=\lceil \nu/m \rceil$.  Form a partition $\Pi$ of the set of $\nu=a\kappa+b$ points into $s$ subsets, including as many as possible of size $m$ and (unless $\kappa$ divides $\tau$) one of smaller size.  By the division algorithm, $\kappa=m\tau+\rho$, where $0 \leq \rho \leq \tau-1$.  Then we have
\[ s= \left\lceil \frac{\nu}{m} \right\rceil = \left\lceil \frac{a\kappa+b}{m} \right\rceil = \left\lceil \frac{a(m\tau+\rho)+b}{m} \right\rceil = a\tau + \left\lceil \frac{a\rho+b}{m} \right\rceil = a\tau + \left\lceil \frac{\tau(a\rho+b)}{\kappa-\rho} \right\rceil . \]
The quantity $\displaystyle \frac{\tau(a\rho+b)}{\kappa-\rho}$ is zero if and only if both $b=0$ and $\rho=0$; furthermore, since $a$, $b$ and $\tau$ are constants and $\rho<\tau$, we have that\
\[ \lim_{k\to\infty} \frac{\tau(a\rho+b)}{\kappa-\rho} = 0. \]
Thus, by the definition of a limit, there exists some $\kappa_1$ such that for all $\kappa\geq\kappa_1$, $\displaystyle \frac{\tau(a\rho+b)}{\kappa-\rho} \leq 1$, and therefore $s\in \{ a\tau, a\tau+1 \}$.

We form a covering design as follows: for any combination of $\tau$ of the sets in $\Pi$, take any $\kappa$-subset of points which contains their union.  Then this collection of $\kappa$-subsets forms the blocks of an $(a\kappa+b,\kappa,\tau)$-covering design: the number of blocks is at most $\binom{a\tau+1}{\tau}$; this depends only on the constants $a$ and $\tau$.
\end{proof}

\begin{example} \label{example:petrov}
We will use Proposition~\ref{prop:petrov} to construct a $(23,10,2)$-covering design.  The Sch\"onheim bound gives $L(23,10,2)=7$, while the best-known covering has size $8$ (see~\cite{lajolla}); we can obtain a covering of size 10 as follows.  First, we partition the set $\{1,\ldots,23\}$ into $\lceil 23/ (10/2) \rceil =5$ subsets, four of which have size $10/2=5$ and the remaining set has size $3$:
\[ 1\ 2\ 3\ 4\ 5 \mid 6\ 7\ 8\ 9\ 10 \mid 11\ 12\ 13\ 14\ 15 \mid 16\ 17\ 18\ 19\ 20 \mid 21\ 22\ 23\ \]
Then we form the blocks of our covering by taking the unions of any pair of these subsets, adding extra points arbitrarily if required.  The blocks are the rows of the array below (where $\ast$ indicates symbols which can be replaced arbitrarily):
\[ \begin{array}{cccccccccc}
 1 &  2 &  3 &  4 &  5 &  6 &  7 &  8 &  9 & 10 \\
 1 &  2 &  3 &  4 &  5 & 11 & 12 & 13 & 14 & 15 \\
 1 &  2 &  3 &  4 &  5 & 16 & 17 & 18 & 19 & 20 \\
 1 &  2 &  3 &  4 &  5 & 21 & 22 & 23 & \ast & \ast \\
 6 &  7 &  8 &  9 & 10 & 11 & 12 & 13 & 14 & 15 \\
 6 &  7 &  8 &  9 & 10 & 16 & 17 & 18 & 19 & 20 \\
 6 &  7 &  8 &  9 & 10 & 21 & 22 & 23 & \ast & \ast \\
11 & 12 & 13 & 14 & 15 & 16 & 17 & 18 & 19 & 20 \\
11 & 12 & 13 & 14 & 15 & 21 & 22 & 23 & \ast & \ast \\
16 & 17 & 18 & 19 & 20 & 21 & 22 & 23 & \ast & \ast
\end{array}\]
In fact, for any sufficiently large value of $\kappa$, Proposition~\ref{prop:petrov} will yield a $(2\kappa+3,\kappa,2)$-covering with $\binom{5}{2}=10$ blocks.
\end{example}

\section{Codes from paths and cycles}
In this section, we show how two straightforward classes of graphs---namely paths and cycles---may be used to obtain families of $(G,k)$-codes to which our decoding algorithm can be applied.  While the parameters for codes obtained from paths and cycles are very similar, the key distinction is that the alphabet size is smaller for codes from cycles than for codes from paths, on account of the diameter of a cycle $C_n$ being (approximately) half that of a path $P_n$.

\subsection{Paths}
Let $P_n$ denote a path on $n$ vertices, which has a $k$-resolving set for $k\leq n-1$, and for $k\geq 3$ has $\dim_k(P_n)=k+1$; see~\cite{yero1} for details.  (We remark that $\dim_k(P_n)=k$ for $k=1$ and $k=2$, but this is of no interest from the perspective of error-correction.)  A $(P_n,k)$-code will therefore have $n$ codewords over an alphabet of size $\diam(P_n)=n-1$, of length $\ell=\dim_k(P_n)=k+1$, and with minimum distance at least $k$, so can correct $r=\lfloor (k-1)/2 \rfloor$ errors.

We note that in the extreme case where $k=n-1$, the $k$-metric dimension is $k+1=n$, and thus every vertex is required in a $k$-resolving set.  This also means that the alphabet size, length of the codewords and number of codewords are all equal, making these codes comparable to Latin squares as permutation codes (but with minimum distance one less).
\begin{example}
Consider a path $P_5$ on $5$ vertices, and let $k=4$.  All $5$ vertices are needed in a $4$-resolving set, and the code obtained is as follows:
\[ \begin{array}{c|ccccc}
\textnormal{Vertex} & \multicolumn{5}{|c}{\textnormal{Codeword}} \\ \hline
1 & 0 & 1 & 2 & 3 & 4 \\
2 & 1 & 0 & 1 & 2 & 3 \\
3 & 2 & 1 & 0 & 1 & 2 \\
4 & 3 & 2 & 1 & 0 & 1 \\
5 & 4 & 3 & 2 & 1 & 0
\end{array} \]
This code has $5$ codewords over the alphabet $\{0,1,2,3,4\}$, and has minimum distance $4$, so can correct $\lfloor (4-1)/2 \rfloor =1$ error.
\end{example}

In order to decode a $(P_n,k)$-code, we will require an uncovering with parameters 
$(k+1, \ell-\lfloor (k-1)/2 \rfloor, r')$, or equivalently a $(k+1, \lfloor (k-1)/2 \rfloor, r')$-covering design.  By letting $m=\lfloor (k-1)/2 \rfloor$, this is either a $(2m+2,m,r')$-covering when $k$ is odd, or a $(2m+3,m,r')$-covering when $k$ is even.  In either case, for a given value of $r'$, we can obtain appropriate covering designs from Proposition~\ref{prop:petrov}.  For example, for any path $P_n$ with $n\geq 23$ may use the uncovering arising from Example~\ref{example:petrov} to correct two errors.

\subsection{Cycles}
For a cycle $C_n$ on $n$ vertices, we must consider the cases where $n$ is odd or even separately; details of $k$-resolvability of cycles were given in~\cite{BeardonRV}.

\subsubsection{Odd cycles}
When $n$ is odd, $C_n$ has a $k$-resolving set for $k\leq n-1$ with $\dim_k(C_n)=k+1$ for all $k$.  A $(C_n,k)$-code will therefore have $n$ codewords of length $\ell=\dim_k(C_n)=k+1$, and with minimum distance $k$, so can correct $r=\lfloor (k-1)/2 \rfloor$ errors; all of these parameters are identical to those for a path $P_n$, but this time the codewords are over an alphabet of size $\diam(C_n)=(n-1)/2$.  As the alphabet size does not affect the uncovering needed for decoding, any uncovering for a $(P_n,k)$-code may also be used for a $(C_n,k)$-code when $n$ is odd.

\begin{example}
Consider a cycle $C_5$ on $5$ vertices (labelled $0,\ldots,4$), and let $k=4$.  All $5$ vertices are needed in a $4$-resolving set, and the code obtained is as follows:
\[ \begin{array}{c|ccccc}
\textnormal{Vertex} & \multicolumn{5}{|c}{\textnormal{Codeword}} \\ \hline
0 & 0 & 1 & 2 & 2 & 1 \\
1 & 1 & 0 & 1 & 2 & 2 \\
2 & 2 & 1 & 0 & 1 & 2 \\
3 & 2 & 2 & 1 & 0 & 1 \\
4 & 1 & 2 & 2 & 1 & 0
\end{array} \]
This has $5$ codewords over the alphabet $\{0,1,2\}$, and has minimum distance $4$, so can therefore correct $\lfloor (4-1)/2 \rfloor =1$ error.
\end{example}
We note that, because the underlying graph is a cycle, any codeword may be obtained from another by a cyclic permutation.

\subsubsection{Even cycles}
When $n$ is even, a little more care is required.  Letting $n=2q$, we have that $C_n$ has a $k$-resolving set for $k\leq n-2=2q-2$, and $\dim_k(C_n)=k+1$ for $k\leq q-1$, or $\dim_k(C_n)=k+2$ for $q\leq k \leq 2q-2=n-2$.  Thus for $k\leq q-1$, the parameters of a $(C_n,k)$-code when $n$ is even (as well as those of the uncovering needed for decoding) are the same as those from an odd cycle (although with an alphabet of size $q=n/2$).  For $q\leq k \leq 2q-2$,  they are slightly different: we have $n$ codewords of length $k+2$ and minimum distance $k$.  For decoding, we will need a
$(k+2, k+2-\lfloor (k-1)/2 \rfloor, r')$-uncovering, or equivalently a $(k+2,\lfloor (k-1)/2 \rfloor, r')$-covering design.  Letting $m=\lfloor (k-1)/2 \rfloor$, we require either a $(2m+3,m,r')$-covering or a $(2m+4,m,r')$-covering if $k$ is odd or even respectively; Proposition~\ref{prop:petrov} is applicable in either case.

\begin{example}
Consider a cycle $C_6$ on $6$ vertices (labelled $0,\ldots,5$), and let $k=4$.  All $6$ vertices are needed in a $4$-resolving set, and the code obtained is as follows:
\[ \begin{array}{c|cccccc}
\textnormal{Vertex} & \multicolumn{6}{|c}{\textnormal{Codeword}} \\ \hline
0 & 0 & 1 & 2 & 3 & 2 & 1 \\
1 & 1 & 0 & 1 & 2 & 3 & 2 \\
2 & 2 & 1 & 0 & 1 & 2 & 3 \\
3 & 3 & 2 & 1 & 0 & 1 & 2 \\
4 & 2 & 3 & 2 & 1 & 0 & 1 \\
5 & 1 & 2 & 3 & 2 & 1 & 0
\end{array} \]
This code has $6$ codewords over the alphabet $\{0,1,2,3\}$, and has minimum distance $4$, so can correct $\lfloor (4-1)/2 \rfloor =1$ error.
\end{example}

\section{Codes from grid graphs}

In this section, we will use the family of grid graphs $P_s\boxprod P_t$, i.e.\ the Cartesian product of the paths $P_s$ and $P_t$, to obtain $(P_s\boxprod P_t,k)$-codes. As we will show, for any grid graph $P_s\boxprod P_t$ of order $st$ and any $k\in \{1,\ldots,s+t-2\}$, we have that $\dim_k(P_s\boxprod P_t)=2k$.  This goes on to provide an interesting infinite family of examples.

\subsection{The $k$-metric dimension of grid graphs}
Suppose that $G$ is the grid graph $P_s\boxprod P_t$, that $s\ge t$, and $U=\{u_1,u_2,\ldots,u_s\}$ and $V=\{v_1,v_2,\ldots,v_t\}$ are the vertex sets of $P_s$ and $P_t$, respectively.

It is clear that, if $G$ is a $k$-metric dimensional graph, then for every positive integer $k'\le k$, $G$ also has a $k'$-metric basis. Next we present a characterization of $k$-metric dimensional graphs, obtained in \cite{yero1}, which will be useful in our work. To do so, we need some additional terminology. Given two vertices $x,y\in V(G)$, we say that the set of {\em distinctive vertices} of $x,y$ is $${\cal D}(x,y)=\{z\in V(G): d_{G}(x,z)\ne d_{G}(y,z)\}.$$

\begin{theorem}[Estrada-Moreno {\em et al.}~\cite{yero1}] \label{theokmetric}
A connected graph $G$ is $k$-metric dimensional if and only if $k=\displaystyle\min_{x,y\in V(G)}\vert {\cal D}(x,y)\vert .$
\end{theorem}

To compute the $k$-metric dimension of a grid graph, we need first to determine for which values of $k$ there exists a $k$-metric basis. This is answered by our next result.

\begin{theorem}
For any $s,t\ge 2$, the graph $G=P_s\boxprod P_t$ is $(s+t-2)$-metric dimensional.
\end{theorem}

\begin{proof}
First, we consider the vertices $(u_2,v_1)$ and $(u_1,v_2)$. Notice that
$${\cal D}_G((u_2,v_1),(u_1,v_2))= (U\times \{v_1\})\cup (\{u_1\}\times V)-\{(u_1,v_1)\}.$$
Thus, $|{\cal D}_G((u_2,v_1),(u_1,v_2))|=s+t-2$ and $G$ is $k$-metric dimensional for some $k\le s+t-2$.

On the other hand, let $(u_i,v_j)$ and $(u_g,v_h)$ be two distinct vertices of $G$. We consider the following cases.\\

\noindent Case 1: $i=g$. Hence $j\ne h$ and it follows that $U\times\{v_j,v_h\}\subseteq {\cal D}_G((u_i,v_j),(u_g,v_h))$. Also,
$$|(\{u_i\}\times V)\cap {\cal D}_G((u_i,v_j),(u_g,v_h))|\ge t-1.$$
Thus we have
\begin{align*}
|{\cal D}_G((u_i,v_j),(u_g,v_h))|&\ge |U\times\{v_j,v_h\}|+|(\{u_i\}\times V)\cap {\cal D}_G((u_i,v_j),(u_g,v_h))|-2\\
&\ge 2s+t-3\\
&\ge s+t-1.
\end{align*}

\noindent Case 2: $j=h$. Analogous to Case 1 above, we obtain that $|{\cal D}_G((u_i,v_j),(u_g,v_h))|\ge s+t-1.$\\

\noindent Case 3: $i\ne g$ and $j\ne h$. We may assume that $i<g$. Hence we have one of the following situations.
\begin{itemize}
\item If $j<h$, then we notice that at most two vertices of the set
$(\{u_1,\ldots,u_{g}\}\times \{v_j\})\cup (\{u_{i},\ldots,u_s\}\times \{v_h\})\cup (\{u_i\}\times \{v_1,\ldots,v_{h}\})\cup (\{u_g\}\times \{v_{j},\ldots,v_t\})$ do not belong to the set ${\cal D}_G((u_i,v_j),(u_g,v_h))$ (see Figure \ref{fig-Case-1} for an example with $i=3$, $j=2$, $g=9$ and $h=6$).

\begin{figure}[t]
\centering
\begin{tikzpicture}[scale=.3]
  \foreach \x in {0,2,4,6,8,10,12,14,16,18,20,22}
  {
    \foreach \y in {0,2,4,6,8,10,12,14}
    {
      \node at (\x,\y) [draw, shape=circle,scale=.5] {};
    }
  }
  \foreach \y in {2,4,6,8,10,12,14}
    {
      \node at (16,\y) [draw, fill=black, shape=circle,scale=.5] {};
    }
    \foreach \y in {0,2,4,6,8,10}
    {
      \node at (4,\y) [draw, fill=black, shape=circle,scale=.5] {};
    }
    \foreach \x in {0,2,4,6,8,10,12,14,16}
    {
     \node at (\x,2) [draw, shape=circle,fill=black, scale=.5] {};
    }
    \foreach \x in {4,6,8,10,12,14,16,18,20,22}
    {
     \node at (\x,10) [draw, shape=circle,fill=black, scale=.5] {};
    }
  \foreach \x in {1,2,3,4,5,6}
  {
   \node at (\x+\x-2,-2) {$u_{\x}$};
  }
  \foreach \x in {7,8,9,10,11,12}
  {
   \node at (\x+\x-2,-2) {$u_{\x}$};
  }
  \foreach \y in {1,2,3,4,5,6,7,8}
  {
   \node at (-2,\y+\y-2) {$v_{\y}$};
  }
\node at (14,2) [rectangle,fill=gray, scale=.85] {};
\node at (6,10) [rectangle, fill=gray,scale=.85] {};
\node at (4,2) [rectangle, fill=black,scale=.85] {};
\node at (16,10) [rectangle, fill=black,scale=.85] {};

\end{tikzpicture}
\caption{A sketch of the graph $P_{12}\boxprod P_8$ (edges have not been drawn). Square bolded vertices have the same distance to the square grey vertices, but different distances to the circular bolded vertices.}\label{fig-Case-1}
\end{figure}
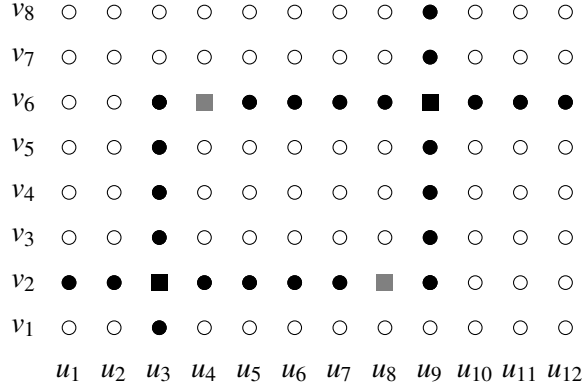

Consequently,
\begin{align*}
|{\cal D}_G((u_i,v_j),(u_g,v_h))|&\ge |\{u_1,\ldots,u_{g}\}\times \{v_j\}|+|\{u_{i},\ldots,u_s\}\times \{v_h\}|+\\
&\hspace*{0.5cm}+|\{u_i\}\times \{v_1,\ldots,v_{h}\}|+|\{u_g\}\times \{v_{j},\ldots,v_t\}|-6\\ 
& \ge g+s-i+1+h+t-j+1-6\\ 
&= s+t+g+h-i-j-4\\ 
&\ge s+t+i+j-i-j-2 \quad \mbox{(since $i<g$ and $j<h$)} \\
&=s+t-2.
\end{align*}
\item If $j>h$, then a similar procedure yields $|{\cal D}_G((u_i,v_j),(u_g,v_h))|\ge s+t-2$.
\end{itemize}
As a consequence, we obtain that $G$ is $(s+t-2)$-metric dimensional.
\end{proof}

Having established that $G=P_s\boxprod P_t$ is $(s+t-2)$-metric dimensional, next we obtain its $k$-metric dimension for every $k\in \{1,\ldots,s+t-2\}$.
As an observation regarding the set of distinctive vertices of a pair of vertices $x,y$, we notice that if $S$ is a $k$-resolving set for a graph $G$, then $|{\cal D}(x,y)\cap S|\ge k$. This simple fact will be used frequently in our next proof.

\begin{theorem}
For any grid graph $G=P_s\boxprod P_t$ and every $k\in \{1,\ldots,s+t-2\}$,
$$\dim_k(G)=2k.$$
\end{theorem}

\begin{proof}
If $k=1$, then it is already known (see \cite{Chartrand00}) that $\dim_1(G)=\dim(G)=2$, so from now on we consider only $k\ge 2$.  Without loss of generality, we suppose that $s\ge t$.  Let $(u_i,v_j)$ and $(u_g,v_h)$ be two distinct vertices of $G$. Next we consider the following three cases according to the value of $k$.\\

\noindent Case 1: $k\le s$. Let $S=\{u_1,\dots,u_k\}\times \{v_1,v_t\}$. We consider the following subcases.\\

\noindent Subcase 1.1: $i=g$. Hence $j\ne h$, and for any vertex $(u_p,v_q)\in S$ it follows $d_{G}((u_i,v_j),(u_p,v_q))\ne d_{G}((u_g,v_h),(u_p,v_q))$. Thus, $(u_i,v_j)$ and $(u_g,v_h)$ are distinguished by $2k$ vertices of $S$.\\

\noindent Subcase 1.2: $i\ne g$. Without loss of generality we assume that $i<g$. Moreover, we consider $j\le h$. First notice that ${\cal D}_G((u_i,v_j),(u_g,v_h))\supseteq \{u_1,\dots,u_i\}\times \{v_1\}$. If $i\ge k$, then clearly $|{\cal D}_G((u_i,v_j),(u_g,v_h))\cap S|\ge k$. Hence, we may assume $i<k$. If $g\ge k$, then there is at most one vertex $(u_f,v_1)\in \{u_{i+1},\dots,u_k\}\times \{v_1\}$ such that $d_{G}((u_i,v_j),(u_f,v_1))= d_{G}((u_g,v_h),(u_f,v_1))$, which leads to
$$|{\cal D}_G((u_i,v_j),(u_g,v_h))\cap (\{u_{1},\dots,u_k\}\times \{v_1\})|\ge k-1.$$
On the other hand, since $i<k$, there is at least one vertex $(u_d,v_t)\in \{u_1,\dots,u_k\}\times \{v_t\}$ such that $d_{G}((u_i,v_j),(u_d,v_t))\ne d_{G}((u_g,v_h),(u_d,v_t))$. As a consequence,
\begin{align*}
  |{\cal D}_G((u_i,v_j),(u_g,v_h))\cap S|& = |{\cal D}_G((u_i,v_j),(u_g,v_h))\cap (\{u_{1},\dots,u_k\}\times \{v_1\})| \\
   & \hspace*{0.5cm} + |{\cal D}_G((u_i,v_j),(u_g,v_h))\cap (\{u_{1},\dots,u_k\}\times \{v_t\})| \\
   & \ge k.
\end{align*}
We now consider $g < k$. Thus, ${\cal D}_G((u_i,v_j),(u_g,v_h))\supseteq \{u_g,\dots,u_k\}\times \{v_t\}$ and at most one vertex
in $\{u_1,\dots,u_g\}\times \{v_1\}$ does not belong to ${\cal D}_G((u_i,v_j),(u_g,v_h))$.
Moreover, ${\cal D}_G((u_i,v_j),(u_g,v_h))$ contains at least one vertex in $\{u_i,\dots,u_g\}\times \{v_t\}$.
As a consequence,
\begin{align*}
  |{\cal D}_G((u_i,v_j),(u_g,v_h))\cap S|& \ge |{\cal D}_G((u_i,v_j),(u_g,v_h))\cap (\{u_1,\dots,u_g\}\times \{v_1\})| \\
   & \hspace*{0.5cm} + |{\cal D}_G((u_i,v_j),(u_g,v_h))\cap (\{u_i,\dots,u_g\}\times \{v_t\}| \\
	 & \hspace*{1cm} + |\{u_g,\dots,u_k\}\times \{v_t\}| \\
   & \ge g-1+1 + k-g+1 \\
   & > k.
\end{align*}
See Figure \ref{fig-Case-1-2} for examples of the situations above considering $k=7$.

\begin{figure}[!htb]
\centering
\begin{tikzpicture}[scale=.3]
  \draw[rounded corners, fill=lightgray] (-1.1,-1.1) rectangle (11.1,1.1);
  \draw[rounded corners, fill=lightgray] (-1.1, 12.9) rectangle (7.1,15.1);
  \draw[rounded corners, fill=lightgray] (8.9, 12.9) rectangle (13.1,15.1);
  \draw[rounded corners, pattern=north east lines, pattern color=gray] (-0.8,-0.8) rectangle (2.8,0.8);
  \draw[rounded corners, pattern=north east lines, pattern color=gray] (3.2,13.2) rectangle (12.8,14.8);
  \draw[rounded corners, pattern=north west lines, pattern color=gray] (-0.8,-0.8) rectangle (12.8,0.8);

  \foreach \x in {0,2,4,6,8,10,12,14,16,18,20,22}
  {
    \foreach \y in {0,2,4,6,8,10,12,14}
    {
      \node at (\x,\y) [draw, shape=circle,scale=.5] {};
    }
  }
    \foreach \x in {0,2,4,6,8,10,12}
    {
     \node at (\x,0) [draw, shape=circle,fill=gray, scale=.6] {};
     \node at (\x,14) [draw, shape=circle,fill=gray, scale=.6] {};
    }

  \node at (2,4) [rectangle,draw, fill=gray, scale=1] {};
  \node at (14,4) [rectangle,draw, pattern=north west lines, scale=1] {};
  \node at (18,6) [rectangle,draw, pattern=north west lines, scale=1] {};
  \node at (6,10) [rectangle,draw, pattern=north east lines, scale=1] {};
  \node at (2,6) [rectangle,draw, pattern=north east lines, scale=1] {};
  \node at (18,8) [rectangle,draw,fill=gray, scale=1] {};

  \foreach \x in {1,2,3,4,5,6}
  {
   \node at (\x+\x-2,-2) {$u_{\x}$};
  }
  \foreach \x in {7,8,9,10,11,12}
  {
   \node at (\x+\x-2,-2) {$u_{\x}$};
  }
  \foreach \y in {1,2,3,4,5,6,7,8}
  {
   \node at (-2,\y+\y-2) {$v_{\y}$};
  }
\end{tikzpicture}
\caption{A sketch of the graph $P_{12}\Box P_8$ (edges have not been drawn). Each pair of square vertices with identically filled shapes is recognized by vertices in the rounded rectangles with identically filled areas, respectively. }\label{fig-Case-1-2}
\end{figure}
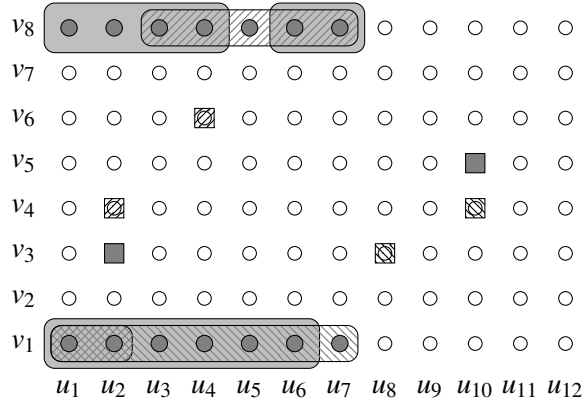

Finally, if $j>h$, then a similar procedure gives an analogous result and we observe that $S$ is a $k$-resolving set for $G$.\\

\noindent Case 2: $s<k\le s+t-2$. Note that $s\ge 3$, since the case $s=2$ would lead to the graph $C_4$, which is $2$-metric dimensional. Assume that $k=s+\alpha$ for some integer $\alpha\ge 1$. Let $S=(U\times\{v_1,v_t\})\cup (\{u_1,u_s\}\times\{v_2,\ldots,v_{\alpha+1}\})$. According to Case 1, the set $S_1= U\times\{v_1,v_t\}\subset S$ is a $s$-resolving set of $G$, which means that any pair of vertices of $G$ is recognized by at least $s$ vertices of $S_1$. On the other hand, also by Case 1 (in a similar version), the set $S_2=\{u_1,u_s\}\times\{v_2,\ldots,v_{\alpha+1}\}\subset S$ is an $(\alpha+1)$-resolving set of $G$, which similarly means any pair of vertices of $G$ is recognized by at least $\alpha+1$ vertices of $S_2$. Since only two vertices, \emph{i.e.} $(u_1,v_1)$ and $(u_s,v_1)$, belong to both sets $S_1$ and $S_2$, it must happen that at least $s+\alpha+1-2$ vertices of $S$ recognize each pair of vertices of $G$. Suppose that there is a pair of vertices  $(u_i,v_j)$ and $(u_g,v_h)$ such that they are distinguished by exactly $s+\alpha-1$ vertices of $S$. Thus, they are distinguished by $(u_1,v_1)$ and $(u_s,v_1)$, by $s-2$ vertices in $S_1\setminus \{(u_1,v_1),(u_s,v_1)\}$ and by $\alpha-1$ vertices in $S_2\setminus \{(u_1,v_1),(u_s,v_1)\}$. However, we can now notice that if $(u_i,v_j)$ and $(u_g,v_h)$ are distinguished by $(u_1,v_1)$ and $(u_s,v_1)$, then $(u_i,v_j)$ and $(u_g,v_h)$ are not distinguished by at most one vertex in $U\times\{v_1\}$ and by at most one vertex in $U\times\{v_t\}$, that is, $2s-2>s$ (since $s\ge 3$) vertices of $S$, which is a contradiction. Thus, each pair of vertices of $G$ is distinguished by at least $s+\alpha=k$ vertices of $S$. Therefore, $S$ is a $k$-resolving set of $G$. As a consequence of both cases, we obtain that $\dim_k(G)\le 2k$.

We now want to show that no smaller $k$-resolving set can exist.  Let $S'$ be a \mbox{$k$-metric} basis for $G$. We consider the vertices $(u_1,v_2)$, $(u_2,v_1)$, $(u_s,v_{t-1})$ and $(u_{s-1},v_t)$. Since
$$|A|=|{\cal D}_G((u_1,v_2),(u_2,v_1))\cap S'|=|(U\times \{v_1\})\cup (\{u_1\}\times V)-\{(u_1,v_1)\}\cap S'|\ge k$$
and
$$|B|=|{\cal D}_G((u_s,v_{t-1}),(u_{s-1},v_t))\cap S'|=|(U\times \{v_t\})\cup (\{u_s\}\times V)-\{(u_s,v_t)\}\cap S'|\ge k$$
we have that
$$|S'|\ge |A|+|B|-2\ge 2k-2.$$
Now suppose $|S'|=2k-2$. Thus, it follows $|A|=k$, $|B|=k$, $(u_s,v_1),(u_1,v_t)\in S$ and $S=A\cup B$. So $|(\{u_s\}\times \{v_2,\ldots,v_{t-1}\})\cup (\{u_2,\ldots,u_{s-1}\}\times \{v_t\})\cap S'|=k-1$ and $|(\{u_1\}\times \{v_2,\ldots,v_{t-1}\})\cup (\{u_2,\ldots,u_{s-1}\}\times \{v_1\})\cap S'|=k-1$. We now consider the vertices $(u_1,v_1)$, $(u_2,v_2)$, $(u_{s-1},v_{t-1})$ and $(u_s,v_t)$.
Hence, if we denote $Q={\cal D}_G((u_1,v_1),(u_2,v_2))$, then
\begin{align*}
|Q\cap S'|&\le |((\{u_s\}\times \{v_2,\ldots,v_{t-1}\})\cup (\{u_2,\ldots,u_{s-1}\}\times \{v_t\}))\cap S'|\\
&= k-1
\end{align*}
and
\begin{align*}
|Q\cap S'|&\le |(\{u_1\}\times \{v_2,\ldots,v_{t-1}\})\cup (\{u_2,\ldots,u_{s-1}\}\times \{v_1\})\cap S'|\\
&= k-1,
\end{align*}
which is a contradiction. So $|S'|\ge 2k-1$. If we suppose that $|S'|=2k-1$, then an analogous procedure to the one above gives a contradiction again. Therefore, we have that $|S'|\ge 2k$ and the proof is complete.
\end{proof}

\subsection{Decoding}

For the grid graphs $G=P_s \boxprod P_t$ described above, and for any $k\leq s+t-2$, we obtain $(G,k)$-codes with $n=st$ codewords of length $\ell=2k$ over an alphabet of size $s+t+1$.  Such a $(G,k)$-code has correction capability $r=\left\lfloor (k-1)/2 \right\rfloor$.  To decode $r'<r$ errors using the algorithm described above will require a $(2k,2k-r,r')$-uncovering, or equivalently a $(2k,r,r')$-covering design.  Depending on the parity of $k$, we therefore require either: (i) if $k=2m$ is even, a $(4m,m-1,r')$-covering design; or (ii) if $k=2m+1$ is odd, a $(4m+2,m,r')$-covering design.  Proposition~\ref{prop:petrov} provides the coverings (and thus uncoverings) we require.  Given that these uncoverings have constant size, the complexity of the decoding algorithm in this case is $O(kn)$.

\section{Conclusion}

The main achievement of this paper was to obtain a new application of $k$-resolving sets in graphs to coding theory, by obtaining a new method of constructing error-correcting codes.  We considered three families of graphs, namely paths, cycles and grid graphs; naturally, there are many more graph families which could be investigated with this application in mind.

We conclude by mentioning that, while $1$-resolving sets have many applications, for the application of $k$-resolving sets to error-correcting codes we require that $k\geq 3$; however, $2$-resolving sets could potentially be applied to the detection of errors.

\end{document}